\documentclass[11pt]{siamart250211}

\usepackage{amsfonts,amsmath,amssymb,bbm}
\usepackage{mathrsfs,mathtools,stmaryrd,wasysym}
\usepackage{xspace,mydef}
\usepackage{esint}
\usepackage{comment}
\usepackage[shortlabels]{enumitem}
\DeclareMathAlphabet{\mathpzc}{OT1}{pzc}{m}{it}

\newcommand{\TheTitle}{Finite element approximation to linear, second order, parabolic problems with $L^1$ data}
\newcommand{\ShortTitle}{Renormalized solutions}
\newcommand{\TheAuthors}{A.J.~Salgado and G.R.~Barrenechea}

\headers{\ShortTitle}{\TheAuthors}

\title{\TheTitle}

\author{Gabriel R.~Barrenechea\thanks{Department of Mathematics and Statistics, University of Strathclyde, 26 Richmond Street, Glasgow G1 1XH, Scotland.
  (\email{gabriel.barrenechea@strath.ac.uk}, \url{https://www.gabrielbarrenechea.com/})}
\and
  Abner J.~Salgado\thanks{Department of Mathematics, University of Tennessee, Knoxville, TN 37996, USA.
    (\email{asalgad1@utk.edu}, \url{https://math.utk.edu/people/abner-salgado/})}
}

\ifpdf
\hypersetup{
  pdftitle={\TheTitle},
  pdfauthor={\TheAuthors}
}
\fi

\begin{document}

\maketitle

\begin{abstract}
We consider the approximation to the solution of the initial boundary value problem for the heat equation with right hand side and initial condition that merely belong to $L^1$. Due to the low integrability of the data, to guarantee well-posedness, we must understand solutions in the renormalized sense. We prove that, under an inverse CFL condition, the solution of the standard implicit Euler scheme with mass lumping converges, in $L^\infty(0,T;L^1(\Omega))$ and $L^q(0,T;W^{1,q}_0(\Omega))$ ($q<\tfrac{d+2}{d+1}$), to the renormalized solution of the problem.
\end{abstract}



\section{Introduction}
\label{sec:Into}

The purpose of this work is to study the convergence properties of a standard finite element discretization to the following initial boundary value problem
\begin{equation}
\label{eq:TheEqn}
  \begin{dcases}
    \partial_t u - \LAP u = f, & \text{ in } (0,T) \times \Omega, \\
    u = 0, & \text{ on } \partial\Omega \times (0,T), \\
    u|_{t=0} = u_0, & \text{ in } \Omega.
  \end{dcases}
\end{equation}
Here $T>0$ is a positive final time and, for $d \geq 1$, $\Omega \subset \Reald$ is a bounded polytope with Lipschitz boundary. The main source of difficulty and originality in our work comes from the data. Namely, we merely assume that the initial condition satisfies $u_0 \in L^1(\Omega)$, and the right hand side is such that $f \in L^1(Q_T)$; see \cref{sec:Notation} for notation.

The limited integrability of the initial data and right hand side prevent \cref{eq:TheEqn} to be understood in the weak setting where, according to \cite[Chapter XVIII]{MR1156075}, one must assume that, at least, $u_0 \in \Ldeux$ and $f \in L^1( 0, T; \Ldeux ) + L^2( 0, T; \Hmun )$. Nevertheless problem \cref{eq:TheEqn} with data in $L^1$ appears, for instance, in the study of the thermistor problem \cite{MR1278895,MR2178229},  or more generally in the modelling of induction heating \cite{CT97}, some Vlasov-Poisson systems \cite{BBC16}, and the modeling of turbulent flows (see, e.g., \cite{GLLMT03}, \cite[Chapter~7]{Chac-Lewa-Book}, and the references therein).
 In order to obtain a satisfactory theory, the notion of renormalized solutions was developed. We refer the reader to \cite{MR1354907,MR1760541} for definitions and results in the elliptic case. This notion was introduced and developed in \cite{MR1489429,MR1876648,MR1453181}. Existence, uniqueness, and stability of solutions was established; as well as, for linear problems, its equivalence with other notions of solution, like that of entropy solutions \cite{MR1894844,MR1900892,MR1436364}. Below, in \cref{sub:DefofRenormalized}, we give a precise definition of renormalized solutions, as well as a summary of their properties.

The nonstandard notion of solution that is needed for a successful PDE theory forces either the development of new numerical schemes, or the reevaluation of existing ones. In this regard we mention, for instance, \cite{MR4353558} which,  after reformulating the PDE as a nonlinear problem using a change of variables, develops a nonlinear finite element method for a linear elliptic problem with $L^1$ data. Finite volume schemes for elliptic \cite{MR3589343,MR4769158} and parabolic \cite{MR2904585} problems with $L^1$ data have been developed. It is shown that these methods converge to a distributional solution. To our knowledge, the only references that deal with renormalized solutions in a numerical setting are \cite{MR2266830,MR3615888}. These show, under certain mesh assumptions, that a standard and a nonlinearly stabilized finite element scheme (the PSI scheme, to be precise) converge to the renormalized solution of an elliptic boundary value problem.

This brings us to the main objective of our work. We study a standard discretization of \cref{eq:TheEqn}: in space it is a piecewise linear, continuous, finite element discretization, whereas in time it is the implicit Euler scheme with mass lumping. We show that, under an inverse CFL condition, see \cref{eq:ReverseCFL}; and the assumption that the underlying spatial meshes support a discrete comparison principle, see \cref{assume:DMP}; the family of numerical solutions converges to the renormalized solution of \cref{eq:TheEqn}. In passing, we prove a conditional inf-sup stability of the implicit Euler method with mass lumping, a result that may be of interest on its own.

To achieve the main objective of our work we organize our presentation as follows. \Cref{sec:Notation} introduces notation and some properties of the truncation operator. In addition, \cref{sub:DefofRenormalized} presents the definition of a renormalized solution and its main properties; namely existence, uniqueness, consistency, and stability. The discussion of the discrete setting begins in \cref{sec:Discretization}, where \cref{assume:DMP} is introduced and detailed. In addition, we recall some properties of the mass lumped inner product. The time discretization is then detailed in \cref{sub:dgTime}. The main technical tools that will be used to prove convergence are the ``\emph{space-time weak-$L^p$}'' estimates of \cref{sub:Technical}. The numerical scheme and its analysis are presented in \cref{sec:TheScheme}. The final technical tool needed for convergence is a conditional inf-sup stability for the implicit Euler scheme with mass lumping. This is discussed in \cref{sub:InfSup}. The analysis of the scheme, \emph{per se}, begins in \cref{sub:Apriori}. Here we provide some useful \emph{a priori} estimates on discrete solutions which, in \cref{sub:Convergence}, serve as basis to assert convergence of our numerical scheme.
Finally, in \cref{sec:OtherStuff} we draw conclusions, some extensions, and avenues of future work.

\section{Notation and preliminaries}
\label{sec:Notation}

We begin by introducing a few relations that will be used throughout our work. $A \coloneqq B$ means equality by definition. $A \lesssim B$ means $A \leq c B$ for a nonessential constant $c$ that may change at each occurrence. $A \gtrsim B$ means $B \lesssim A$, and $A \eqsim B$ is shorthand for $A \lesssim B \lesssim A$.

The spatial dimension shall be denoted by $d \in \Natural$. The spatial domain is $\Omega \subset \Reald$ and it will be assumed to be a bounded polytope with Lipschitz boundary. The assumption that the domain is a polytope is done merely for convenience; essentially, so that it can be meshed exactly. By $T>0$ we denote our final time, and the space-time cylinder shall be denoted by
\[
  Q_T \coloneqq (0,T) \times \Omega.
\]

We shall adhere to standard notation regarding function spaces. Thus, symbols like $H^1_0(\Omega)$, $L^1(Q_T)$, or $L^2(0,T;L^3(\Omega))$ carry the expected meaning (see, e.g., \cite{MR4242224} for the notation).  In the case of vector-valued variables, the function spaces will be denoted by boldface letters.  In addition, if $N \in \Natural$ and $E \subset \Real^N$, we shall denote by $|E|$ its $N$-dimensional Lebesgue measure. By $L^0(E)$ we denote the collection of measurable, and almost everywhere finite functions $E \subset \Real^N \to \bar{\Real}$. For $p \in [1,\infty)$ the Marcinkiewicz or weak-$L^p$ space is
\[
  L^{p,\infty}(E) \coloneqq \left\{ w \in L^0(E) : \sup_{\lambda>0} \lambda^p \left| \left\{ z \in E : |w(z)| > \lambda \right\} \right| < \infty \right\}
\]
with norm
\[
  \| w \|_{L^{p,\infty}(E)} \coloneqq \sup_{\lambda>0} \lambda \left| \left\{ z \in E : |w(z)| > \lambda \right\} \right|^{1/p}.
\]
We refer the reader to \cite{MR3243734,MR3024912} for properties of these spaces. In particular; see \cite[Exercise 1.1.11]{MR3243734}, \cite[Theorem 3.18.8]{MR3024912}; we have that, if $|E|<\infty$, whenever $r<p$, then $L^{p,\infty}(E) \hookrightarrow L^r(E)$.

Regarding the problem data, we assume that the initial condition is $u_0 \in L^1(\Omega)$; whereas the right hand side $f \in L^1(Q_T)$.

For convenience, we fix a few dimension dependent numbers that will appear repeatedly in our derivations. The first one is the critical exponent in the embedding $\Hunz \hookrightarrow L^s(\Omega)$. Thus, if $d=2$, we let $s >1$ be an arbitrarily large number; whereas, for $d \geq 3$,
\begin{equation}\label{s-definition}
  s \coloneqq \frac{2d}{d-2} > 2.
\end{equation}
Finally, we let
\begin{equation}\label{qbar-definition}
  \barq \coloneqq \frac{d+2}{d+1} <2.
\end{equation}

\subsection{Truncations}
\label{sub:Truncations}

For $k >0$ we define the function $\Trunc_k :\Real \to \Real$ as
\begin{equation}
\label{eq:DefofTk}
  \Trunc_k(s) \coloneqq \min\left\{ k, \max\{ -k, s \} \right\}.
\end{equation}
Since this function is nondecreasing and odd, its primitive
\[
  \Theta_k(s) \coloneqq \int_0^s \Trunc_k (r) \diff r
\]
is convex, even, and, by construction, $\Theta_k(0) = 0$. Observe also that
\begin{equation}
\label{eq:ThetaAndAbs}
  \Theta_1(s) \leq |s| \leq \Theta_1(s) + \frac12, \qquad \forall s \in \Real.
\end{equation}
Finally, see \cite[Theorem A.1]{MR1786735}, we recall that if $w \in \Hunz$ then, for every $k>0$, $\Trunc_kw \coloneqq \Trunc_k \circ w \in \Hunz$ with
\[
  \GRAD \Trunc_kw(x) = \begin{dcases}
                      \GRAD w(x), & x \in \left\{ z \in \Omega: |w(z)| \leq k \right\}, \\
                      \bzero, & x \notin \left\{ z \in \Omega: |w(z)| \leq k \right\}.
                    \end{dcases}
\]

\subsection{Renormalized solutions}
\label{sub:DefofRenormalized}

We are now in position to define the notion of renormalized solution to \cref{eq:TheEqn}. The idea is to test, for a suitable function $\eta : \Real \to \Real$, with $\eta(u) v$, where $v \in C_0^\infty( 0,T ; \Hunz \cap \Linf)$, and integrate by parts.

\begin{definition}[renormalized solution]
\label{defn:Renormalized}
We say that the function
\[
  u \in C([0,T];L^1(\Omega))
\]
is a renormalized solution to \cref{eq:TheEqn} if:
\begin{itemize}[left=0pt]
  \item For every $k>0$, $\Trunc_ku \in L^2(0,T;\Hunz)$.

  \item We have, as $k \to \infty$,
  \[
    \frac1k \int_{Q_T} |\GRAD \Trunc_ku|^2 \diff x \diff t \to 0.
  \]

  \item For every $\eta \in C_0^{0,1}(\Real)$ and all $v \in C_0^\infty( 0,T; \Hunz \cap \Linf)$
  \begin{equation}
  \label{eq:Formulation}
    -\int_{Q_T} N(u) \partial_t v \diff x \diff t + \int_{Q_T} \GRAD u \cdot \GRAD \left( \eta(u) v \right) \diff x \diff t = \int_{Q_T} f \eta(u) v \diff x \diff t,
  \end{equation}
  where $N' = \eta$.

  \item $u(0) = u_0$ in $L^1(\Omega)$.
\end{itemize}
\end{definition}

We immediately comment that \cref{eq:Formulation} requires some explanation. Since $\eta$ has compact support, there is $k>0$ such that $\supp \eta \subset [-k,k]$. Therefore we may rewrite
\begin{align*}
  \GRAD u \cdot \GRAD\left( \eta(u) v \right) &= \GRAD u \cdot \left[ \eta(u) \GRAD v + \eta'(u) v \GRAD u  \right] = \eta(u) \GRAD u \cdot \GRAD v + v\eta'(u) |\GRAD u|^2 \\
    &= \eta( \Trunc_ku ) \GRAD \Trunc_ku \cdot \GRAD v + v\eta'(\Trunc_ku) |\GRAD \Trunc_ku|^2.
\end{align*}
The above calculation justifies why every term in \eqref{eq:Formulation} is meaningful and integrable.

As mentioned in the Introduction, this notion was introduced, for instance, in \cite{MR1489429}. The relevant results regarding renormalized solutions are summarized below. We refer to \cite{MR1489429,MR1876648} for their proofs.

\begin{theorem}[renormalized solutions]
\label{thm:Renormalized}
  Under the running assumptions for $\Omega$, $T$ we have:
  \begin{itemize}[left=0pt]
    \item \textbf{Existence and uniqueness.} For every $(u_0,f) \in L^1(\Omega) \times L^1(Q_T)$ there is a unique renormalized solution to \cref{eq:TheEqn} in the sense of \cref{defn:Renormalized}.

    \item \textbf{Consistency.} If $u \in L^\infty(0,T;\Ldeux) \cap L^2(0,T;\Hunz)$ is a weak solution to \cref{eq:TheEqn}, then it is a renormalized solution. Conversely, if a renormalized solution is sufficiently smooth, then it is also a weak solution.

    \item \textbf{Stability and continuous dependence:} Let $\{ (u_{0,m},f_m) \}_{m \in \Natural} \subset L^1(\Omega) \times L^1(Q_T)$ and denote by $\{u_m\}_{m \in \Natural}$ the corresponding family of renormalized solutions. If, as $m \to \infty$, we have that
    \[
      (u_{0,m},f_m) \to (u_{0},f)
    \]
    in $L^1(\Omega) \times L^1(Q_T)$, then there is a function
    \[
      u \in C([0,T];L^1(\Omega) ) \cap L^q(0,T;W^{1,q}_0(\Omega)), \qquad q<\barq,
    \]
    such that $u_m \to u$ in $L^\infty(0,T;L^1(\Omega) ) \cap L^q(0,T;W^{1,q}_0(\Omega))$, and $u$ is a renormalized solution to \cref{eq:TheEqn}, in the sense of \cref{defn:Renormalized}.
  \end{itemize}
\end{theorem}

\section{Discretization}
\label{sec:Discretization}

Let us now describe the numerical scheme that we will employ. In essence we will consider a $dG(0)$-in-time and $\polP_1$ in space discretization. In our description, we will adhere to established notation and lexicon; see \cite{MR4242224,MR4269305,MR4248810} for context.

\subsection{Spatial discretization}
\label{sub:FEM}
We begin with the spatial discretization. Since it is assumed that $\Omega$ is a polytope, it can be meshed exactly. We let $\{\Triang \}_{h >0}$ denote a conforming and quasiuniform family of simplicial triangulations of $\bar\Omega$ parametrized by $h >0$, which denotes the mesh size. By $\{\Fespace\}_{h > 0}$ we denote the ensuing family of finite element spaces, \ie
\[
  \Fespace \coloneqq \left\{ w_h \in C(\bar\Omega) : w_{h|T} \in \polP_1, \ \forall T \in \Triang \;,\; w_{h|\partial\Omega}=0\right\}.
\]
Given $h > 0$ we denote by $\Vert$ the collection of vertices of $\Triang$, $\VertInt = \Vert \cap \Omega$, and $\VertBd = \Vert \cap \partial\Omega$. The canonical basis of $\Fespace$ is denoted by $\{\phi_\vertex\}_{\vertex \in \VertInt}$. The Lagrange interpolant $\calL_h : C(\bar\Omega) \to \Fespace$ is defined as
\[
  \calL_h w (x) = \sum_{\vertex \in \VertInt} w(\vertex) \phi_\vertex(x).
\]
The $L^2$-projection $\calP_h : L^1(\Omega) \to \Fespace$ is defined as
\[
  \int_\Omega (w - \calP_h w) \varphi_h \diff x = 0, \qquad \forall \varphi_h \in \Fespace.
\]
We recall that since the family of meshes is assumed to be quasiuniform, see \cite[Theorem 4.14]{MR4320894}, $\calP_h$ is stable in $L^1$, \ie
\begin{equation}
\label{eq:L1stabL2proj}
\| \calP_h w \|_{L^1(\Omega)} \leq C_\calP \| w \|_{L^1(\Omega)}, \qquad \forall w \in L^1(\Omega).
\end{equation}

For our constructions, it is necessary to assume that our mesh supports a \emph{discrete comparison principle}. Namely, we require that a version of \cite[Lemma 11]{MR3073956} holds.

\begin{assumption}[DMP]
\label{assume:DMP}
  For every $k>0$ and all $w_h \in \Fespace$ we have
  \[
    \GRAD w_h \cdot \GRAD \calL_h \Trunc_kw_h \geq |\GRAD \calL_h \Trunc_k w_h |^2, \qquad \mae \ \Omega,
  \]
  and, as a consequence,
  \[
    |\GRAD w_h | \geq |\GRAD \calL_h \Trunc_k w_h |, \qquad \mae \ \Omega.
  \]
\end{assumption}

We comment that, as mentioned in \cite{MR3073956}, this property holds whenever the mesh $\Triang$ is \emph{nonobtuse}, meaning that every dihedral angle in the triangulation is smaller than, or equal to, $\tfrac\pi2$, which in turn implies that
\[
  \GRAD \phi_{\vertex} \cdot \GRAD \phi_{\othervertex} \leq 0, \quad \mae \ \Omega, \qquad \forall \vertex, \othervertex \in \VertInt, \ \vertex \neq \othervertex.
\]

\subsection{Mass lumping}
\label{sub:MassLumping}

For $p \in [1,\infty)$ we define the so-called mass lumped $L^p$-norm
\[
  \| v_h \|_{L^p_h}^p \coloneqq \int_\Omega \calL_h \left(  |v_h|^p \right) \diff x, \qquad \forall v_h \in \Fespace.
\]
As expected, the case $p=2$ can be defined from an inner product, namely,
\[
  (v_h, w_h)_{L^2_h} \coloneqq \int_\Omega \calL_h \left( v_hw_h \right) \diff x, \qquad \forall v_h,w_h \in \Fespace,
\]
which we call the mass lumped inner product. Some, simple yet important, properties of this inner product and the $L^p_h$-norms are detailed below.

\begin{proposition}[mass lumping]
\label{prop:MassLump}
  Let $p \in [1,\infty)$. The mass lumped $L^p$-norm satisfies
  \begin{equation}
  \label{eq:ErrEstMassLump}
    \| w_h \|_{L^p(\Omega)} \leq \| w_h\|_{L^p_h} \leq C_p \| w_h \|_{L^p(\Omega)}, \qquad \forall w_h \in \Fespace,
  \end{equation}
  where $C_p$ is independent of $h$. In particular, $C_2 = \sqrt{d+2}$.
  In addition, there is a constant $C_Q$, independent of $h>0$, for which
  \begin{equation}\label{Eq:error-ML}
    \left| (v_h,w_h)_{L^2_h} - \int_\Omega v_h w_h \diff x \right| \leq C_Q h \| v_h \|_{\Ldeux} \| \GRAD w_h \|_{\Ldeuxd}, \quad \forall v_h,w_h \in \Fespace.
  \end{equation}
\end{proposition}
\begin{proof}
 The first inequality in \eqref{eq:ErrEstMassLump} can be easily deduced from the fact that the canonical basis contains only non-negative functions and forms a partition of unity. Thus
  \[
    w_h = \sum_{\vertex \in \VertInt} W_\vertex \phi_\vertex
  \]
  is in fact a convex combination of the numbers $\{W_\vertex\}_{\vertex \in \VertInt} \subset \Real$. Therefore, since the function $\Real \ni s \mapsto |s|^p$ is convex,
  \begin{align*}
    \int_\Omega |w_h|^p \diff x &= \int_\Omega \left| \sum_{\vertex \in \VertInt} W_\vertex \phi_\vertex \right|^p \diff x \leq \int_\Omega \sum_{\vertex \in \VertInt} \left| W_\vertex \right|^p  \phi_\vertex \diff x = \int_\Omega \calL_h(|w_h|^p) \diff x
      \\
      &= \| w_h \|_{L_h^p}^p,
  \end{align*}
  as claimed. The second inequality in \eqref{eq:ErrEstMassLump} is standard in the literature. It follows the proof of condition number estimates for the mass matrix; see, for instance, \cite[Proposition 28.6]{MR4269305}. We refer also to \cite[Lemma 3.9]{MR3309171} for the value of the constant in the second inequality, and for the proof of \eqref{Eq:error-ML}.
\end{proof}

Next we show how the discrete $L^1_h$-norm interacts with the function $\Theta_k$.

\begin{lemma}[nonlinear estimate]
\label{lem:ThetakVSmassLump}
  For every $k \geq 0$ and all $w_h \in \Fespace$ we have
  \[
    \| \Theta_k( w_h ) \|_{L^1_h} \leq C_1 k \| w_h \|_{L^1(\Omega)},
  \]
  where $C_1$ is the constant from \cref{prop:MassLump}.
\end{lemma}
\begin{proof}
  By definition
  \[
     \Theta_k(s) = \begin{dcases}
                    \frac12 s^2, & |s| \leq k, \\
                    k|s| - \frac{k^2}2, & |s|>k.
                  \end{dcases}
  \]
  Therefore, upon defining
  \[
    S_k(w_h) \coloneqq \left\{ \vertex \in \VertInt : |w_h(\vertex)| \leq k \right\}, \qquad
    B_k(w_h) \coloneqq \VertInt \setminus S_k(w_h),
  \]
  we may compute
  \begin{align*}
    \| \Theta_k( w_h ) \|_{L^1_h} &= \int_\Omega \calL_h \Theta_k(w_h) \diff x 
     = \sum_{\vertex \in \VertInt} \Theta_k( w_h(\vertex) ) \int_\Omega \phi_\vertex \diff x
     \\
     &= \frac12 \sum_{\vertex \in S_k(w_h)}  |w_h(\vertex)|^2 \int_\Omega \phi_\vertex \diff x 
     + \sum_{\vertex \in B_k(w_h)} \left( k |w_h(\vertex)| - \frac{k^2}2 \right) \int_\Omega \phi_\vertex \diff x 
     \\
     &\leq \frac12 \sum_{\vertex \in S_k(w_h)} |w_h(\vertex)|^2 \int_\Omega \phi_\vertex \diff x 
     + k \sum_{\vertex \in B_k(w_h)} |w_h(\vertex)| \int_\Omega \phi_\vertex \diff x .
  \end{align*}
  We now use that,
  \[
    |s| \leq k \qquad \implies \qquad s^2 \leq k|s|,
  \]
  to continue our estimate as
  \[
    \| \Theta_k( w_h ) \|_{L^1_h} \leq k \sum_{\vertex \in \VertInt} |w_h(\vertex)| \int_\Omega \phi_\vertex \diff x.
  \]
  Finally, we use \eqref{eq:ErrEstMassLump} to conclude.
\end{proof}

\subsection{Temporal discretization}
\label{sub:dgTime}
We can now describe the temporal discretization. Given $\calN \in \Natural$, we let $\dt = \{t_n\}_{n=0}^\calN$ be a partition of $[0,T]$, \ie
\[
  0 = t_0 <  \cdots < t_\calN = T.
\]
We denote $\tau_n = t_{n} - t_{n-1}$, and $I_n = (t_{n-1},t_n]$. By $\dt >0$ we denote any collection of such temporal partitions. By $\dt \to 0$ we denote 
\[
  \lim_{\calN \to \infty} \max_{n=1,\ldots, \calN} \tau_n = 0.
\]
This could be more rigorously described using nets \cite[\S I.6]{MR1700700}, but we shall not make an attempt to do so.

The space of space-time discrete functions is then defined as
\[
  \frakX_h^\dt \coloneqq \left( \Fespace \right)^{\calN+1},
\]
and understand it as the space of functions $w_h^\dt :[0,T] \to \Fespace$ such that, if $\{w_h^n\}_{n=0}^{\calN} \in \frakX_h^\dt$, then
\[
  w_h^\dt(0) = w_h^0, \qquad w_h^\dt(t) = w_h^n, \ t \in I_n, \quad n = 1, \ldots, \calN.
\]
As usual, $\jump{w_h^\dt }_{n-1} \coloneqq w_h^n - w_h^{n-1}$. Given $w_h^\dt \in \frakX_h^\dt$ its so-called \emph{reconstruction} is the function $\calR^\dt w_h^\dt \in C^{0,1}([0,T];\Fespace)$  defined as
\[
  \calR^\dt w_h^\dt (t) = w_h^{n-1} + \jump{w_h^\dt}_{n-1} \frac{t- t_{n-1}}{\tau_n} , \qquad t \in I_n, \quad n = 1, \ldots, \calN.
\]
Notice that, for all $n = 0, \ldots, \calN$, $\calR^\dt w_h^\dt(t_n) = w_h^\dt(t_n)$  and that
\[
  \partial_t \calR^\dt w_h^\dt (t) = \frac1{\tau_n} \jump{w_h^\dt}_{n-1} , \qquad t \in \mathring{I}_n, \quad n = 1, \ldots, \calN.
\]
We endow the space $\frakX_h^\dt$ with the norm
\begin{align*}
  \| w_h^\dt \|_{\frakX_h^\dt}^2 &\coloneqq \| w_h^\dt \|_{L^2(0,T;\Hunz)}^2 + \| \partial_t \calR^\dt w_h^\dt \|_{L^2(0,T;\Hmun)}^2 + \| w_h^\dt(T) \|_{\Ldeux}^2  \\
    &+ \sum_{n=1}^\calN \left\| \jump{w_h^\dt}_{n-1} \right\|_\Ldeux^2.
\end{align*}

Finally, we let $\frakY_h^\dt \coloneqq \frakX_h^\dt$ algebraically, but normed as
\[
  \| w_h^\dt \|_{\frakY_h^\dt}^2 \coloneqq  \| w_h^\dt(0) \|_\Ldeux^2 + \| w_h^\dt \|_{L^2(0,T;\Hunz)}^2.
\]

\subsection{Some estimates from truncations}
\label{sub:Technical}
Here we present some estimates that shall be useful for our purposes. In a sense, these represent the time-dependent version of those in \cite[Section 2]{MR2266830}, and a discrete version of those in \cite[Section IV]{MR1025884}. These estimates shall be the fundamental tools that will allow us to assert convergence.

We begin by recalling a technical result from \cite{MR2266830}. It essentially asserts that if a finite element function is ``\emph{big}'' at a point, it cannot be ``\emph{too small}'' in the whole element that contains said point.

\begin{lemma}[truncation vs.~interpolation]
\label{lem:Girault}
  Let $k>0$, $w_h \in \Fespace$, and $T \in \Triang$ be such that there is $ y \in T$ for which
  \[
    |w_h(y)| \geq k.
  \]
  Then, there is a subsimplex $S_T \subset T$, with $|S_T| \eqsim |T|$ for which
  \[
    |\calL_h \Trunc_k w_h(x) | \geq \frac{k}2, \qquad \forall x \in S_T.
  \]
\end{lemma}
\begin{proof}
  See \cite[Lemma 2.3]{MR2266830}.
\end{proof}

The following result is the main technical tool of this work.

\begin{theorem}[truncations]
\label{thm:Truncations}
  Assume that $\{w_h^\dt \in \frakX_h^\dt \}_{h>0,\dt >0 }$ is a family of space-time discrete functions for which there are constants $F,U >0$ such that, for every $k>0$,
  \begin{equation}
  \label{eq:MainEst}
    \left\|  \calL_h\Theta_k( w_h^\dt ) \right\|_{L^\infty(0,T;L^1(\Omega))} + \int_0^T \int_\Omega |\GRAD \calL_h\Trunc_kw_h^\dt|^2 \diff x \diff t \leq  k \left( F+ U \right).
  \end{equation}
  Then, recalling that $\barq$ is defined in \cref{qbar-definition}, we have
  \begin{align}
  \label{eq:LinftytL1x}
    \| w_h^\dt \|_{L^\infty(0,T;L^1(\Omega))} &\leq F + U + \frac12|\Omega|,
    \\
  \label{eq:LqWeakSpaceTime}
    \| \GRAD w_h^\dt \|_{\bfL^{\barq,\infty}(Q_T)}^\barq &\lesssim \max\left\{ \left( F + U + \frac12|\Omega| \right)^{2/d}, 1 \right\} \left( F + U\right) ,
    \\
  \label{eq:LsWeakSpaceTime}
    \| w_h^\dt \|_{L^{(d+2)/d,\infty}(Q_T)}^{(d+2)/d} &\lesssim \max\left\{ \left( F + U + \frac12|\Omega| \right)^{2/d}, 1 \right\} \left( F + U \right).
  \end{align}
\end{theorem}
\begin{proof}
  Set, in \cref{eq:MainEst}, $k=1$ to observe that, since $w_h^\dt$ is piecewise constant in time,
  \[
    \max_{n=1,\ldots,\calN} \int_\Omega \calL_h \Theta_1( w_h^n ) \diff x \leq F + U.
  \]
Let $n \in \{1, \ldots, \calN\}$ be arbitrary. By \cref{eq:ThetaAndAbs}, and the fact that $\calL_h$ is order preserving, the previous estimate implies
  \begin{equation}
  \label{eq:LinftL1xDetails}
    \| w_h^n \|_{L^1_h} = \int_\Omega \calL_h |w_h^n| \diff x \leq \int_\Omega \calL_h\left(  \Theta_1( w_h^n ) + \frac12 \right) \diff x \leq F + U + \frac12|\Omega|.
  \end{equation}
  Since $n$ is arbitrary, estimate \eqref{eq:ErrEstMassLump} implies \cref{eq:LinftytL1x}.

  With this at hand, we now obtain an auxiliary estimate.  Let, once again, $n$ be arbitrary. By observing that, for every $k>0$, $|\Trunc_k w_h^n | \leq |w_h^n|$ we may then write
\begin{align*}
    \int_\Omega |\calL_h \Trunc_kw_h^n| \diff x &= \int_\Omega \left| \sum_{\vertex \in \VertInt} \Trunc_kw_h^n(\vertex) \phi_\vertex \right| \diff x \leq \int_\Omega \sum_{\vertex \in \VertInt} |\Trunc_k w_h^n(\vertex)| \phi_\vertex \diff x
     \\
     &\leq \int_\Omega \sum_{\vertex \in \VertInt} | w_h^n(\vertex)| \phi_\vertex \diff x \leq \int_\Omega \calL_h |w_h^n| \diff x \leq F+ U + \frac12|\Omega|,
  \end{align*}
  where, in the last step, we used \eqref{eq:LinftL1xDetails}. Thus, since $n$ was assumed arbitrary,
  \begin{equation}
  \label{eq:ScaryEstimate}
    \max_{n=1,\ldots,\calN}\int_\Omega |\calL_h \Trunc_kw_h^n| \diff x  \leq F+U + \frac12|\Omega|.
  \end{equation}

  We now prove \cref{eq:LqWeakSpaceTime}. Fix $\lambda >0$ and observe that, since $w_h^\dt$ is piecewise constant in time,
  \[
    |\calA(\lambda)| \coloneqq \left| \left\{ (t,x) \in Q_T : |\GRAD w_h^\dt(t,x)| > \lambda \right\} \right| = 
      \sum_{n=1}^\calN \tau_n |\calA_n(\lambda)|,
  \]
  where
  \[
    \calA_n(\lambda) \coloneqq \left\{ x \in \Omega : |\GRAD w_h^n(x)| > \lambda \right\}.
  \]

  We now let $k>0$, to be specified later, and define
  \[
    \calB_n(k) \coloneqq \left\{ T \in \Triang : \exists y \in T  \ |w_h^n(y)| > k \right\}.
  \]
  Since
  \[
    \calA_n(\lambda) = \left( \calA_n(\lambda) \bigcap \cup \calB_n(k) \right) \bigsqcup \left\{ x \notin \cup \calB_n(k) : |\GRAD w_h^n (x) | > \lambda \right\},
  \]
  we have
  \[
    |\calA_n(\lambda)| \leq | \cup \calB_n(k)| + \left| \left\{ x \notin \cup \calB_n(k) : |\GRAD w_h^n (x) | > \lambda \right\} \right| = |\mathrm{I}| + |\mathrm{II}|.
  \]
  We estimate the measure of each set separately.
  
  First we note that, if $T \notin \calB_n(k)$, we have that $|w_h^n(y)| \leq k$ for all $y \in T$. Therefore, for every $x \in T$,
  \[
    \Trunc_k w_h^n(x) = w_h^n(x), \quad \calL_h \Trunc_k w_h^n(x) = \calL_h w_h^n(x) = w_h^n(x) , \quad \GRAD \calL_h \Trunc_k w_h^n(x) = \GRAD w_h^n(x).
  \]
  This, in turn, implies that
  \[
    |\mathrm{II}| \leq \frac1{\lambda^2} \int_{\mathrm{II}} |\GRAD w_h^n|^2 \diff x = \frac1{\lambda^2} \int_{\mathrm{II}} |\GRAD \calL_h \Trunc_k w_h^n|^2 \diff x \leq \frac1{\lambda^2} \int_\Omega |\GRAD \calL_h \Trunc_k w_h^n|^2 \diff x.
  \]

  The estimate of $|\mathrm{I}|$ is more involved. For definiteness we present the argument in the case $d \geq 3$. The arithmancy regarding integrability indices can be easily adjusted for $d=2$. To begin, we define
  \begin{equation}\label{r-definition}
    r = \frac{2(d+1)}d < s,
  \end{equation}
  where we recall that $s$ is defined in \eqref{s-definition}.
  Observe now that, using \cref{lem:Girault},
  \begin{align*}
    |\mathrm{I}| &\leq \sum_{T \in \calB_n(k)} |T| \lesssim \sum_{T \in \calB_n(k)} |S_T| \leq \frac{2^r}{k^r} \sum_{T \in \calB_n(k)}  \int_{S_T} |\calL_h \Trunc_k w_h^n|^r \diff x 
    \\
    &\leq \frac{2^r}{k^r} \int_\Omega |\calL_h \Trunc_k w_h^n|^r \diff x = \frac{2^r}{k^r} \| \calL_h \Trunc_k w_h^n \|_{L^r(\Omega)}^r.
  \end{align*}
  We then apply a well-known interpolation inequality, \cite[Proposition 6.10]{MR1681462}, and \eqref{eq:ScaryEstimate} to assert that
  \[
    |\mathrm{I}| \lesssim \frac{2^r}{k^r} \| \calL_h \Trunc_k w_h^n \|_{L^1(\Omega)}^{\theta r} \| \calL_h \Trunc_k w_h^n \|_{L^s(\Omega)}^{(1-\theta) r} \lesssim \frac{\calM}{k^r} \| \calL_h \Trunc_k w_h^n \|_{L^s(\Omega)}^{(1-\theta) r},
  \]
  where
  \[
    \frac1r = \theta + \frac{1-\theta}s, \qquad \calM \coloneqq \max\left\{ \left( F+U + \frac12|\Omega| \right)^{\theta r}, 1\right\}.
  \]
  Next, we invoke the Sobolev embedding theorem to realize that
  \begin{equation}
  \label{eq:MeasBnk}
    |\mathrm{I}| \lesssim \frac{\calM}{k^r} \| \GRAD \calL_h\Trunc_k w_h^n \|_\Ldeux^{(1-\theta)r}.
  \end{equation}
  Notice that a simple computation reveals that $(1-\theta)r = 2$ and $\theta r = \tfrac2d$. 
  
  We now use these estimates to obtain that, for $\lambda>0$ and $k>0$,
  \[
    |\calA(\lambda)| \lesssim \left[ \frac{\calM}{k^r}  + \frac1{\lambda^2} \right] \sum_{n=1}^\calN \tau_n \| \GRAD \calL_h\Trunc_k w_h^n \|_\Ldeux^2 
    \lesssim \calM \left[ \frac1{k^r}  + \frac1{\lambda^2} \right]  k(F+U),
  \]
  where we also used \cref{eq:MainEst}. Up to this point $k>0$ was arbitrary, we may then set $k = \lambda^{2/r}$ to obtain
  \[
|\calA(\lambda)| \lesssim \calM \lambda^{2/r-2}( F+U ).
  \]
  Observe now that
  \[
    2-\frac2r = \barq,
  \]
  where we recall that $\barq$ is defined in \eqref{qbar-definition}.
  Consequently,
  \[
    \| \GRAD w_h^\dt \|_{\bfL^{\barq,\infty}(Q_T)}^\barq = \sup_{\lambda>0} \lambda^\barq |\calA(\lambda)| 
    \lesssim \calM ( F + U ),
  \]
  as we had intended to show.

  Finally, estimate \cref{eq:LsWeakSpaceTime} is essentially already proved. Indeed, we let $k>0$ be arbitrary and observe that
  \[
    \calC_n(k) \coloneqq \left\{ x \in \Omega : |w_h^n(x)| > k \right\} \subset \cup \calB_n(k).
  \]
  Estimate \cref{eq:MeasBnk} together with \cref{eq:MainEst} then imply that
  \[
  \sum_{n=1}^\calN \tau_n |\calC_n(k) | \lesssim \frac{\calM}{k^r} k(F + U) = \calM (F + U) k^{1-r} .
  \]
  Upon observing that $r-1 = \tfrac{d+2}d$ we then realize that
  \begin{align*}
    \| w_h^\dt \|_{L^{(d+2)/d,\infty}(Q_T)}^{(d+2)/d} &= \sup_{k>0} k^{(d+2)/d} \sum_{n=1}^\calN \tau_n |\calC_n(k) | 
    \lesssim   \calM (F + U).
  \end{align*}
  
  All the estimates have been obtained, and this proves the result.
\end{proof}

\begin{remark}[extension to $p \neq 2$]
The proof of this last result, without much effort, can be easily generalized as follows. If $p \in (2-1/d,d]$ and
\[
  \left\|  \calL_h\Theta_k( w_h^\dt ) \right\|_{L^\infty(0,T;L^1(\Omega))} + \int_0^T \int_\Omega |\GRAD \calL_h \Trunc_kw_h^\dt|^p \diff x \diff t \leq  k (F+ U),
\]
then, for
\[
  \widetilde{q} \coloneqq \frac{p(d+1)-d}{d+1},
\]
we have
\[
  \| \GRAD w_h^\dt \|_{\bfL^{\widetilde{q},\infty}(Q_T)}^{\widetilde{q}} \lesssim F+ U.
\]
This result is of interest by itself, but it is not needed in our analysis below, hence we will not dwell on it.
\end{remark}

\section{The numerical scheme and its analysis}
\label{sec:TheScheme}

We have now reached the point where we are able to present our numerical method. In essence, we employ the mass-lumped implicit Euler scheme. We begin by discretizing the right hand side in time. Namely, we construct $f^\dt = \{f^n\}_{n=1}^\calN \subset L^1(\Omega)$ as
\[
  f^n = \frac1{\tau_n}\int_{I_n} f \diff t.
\]

The numerical scheme constructs $u_h^\dt = \{ u_h^n \}_{n=0}^\calN \in \frakX_h^\dt$ as follows. Let $u_h^0 = \calP_h u_0$. Then, for $n \geq 1$, we compute $u_h^n \in \Fespace$ as the solution to
\begin{equation}
\label{eq:TheScheme}
  \left( \frac{u_h^n - u_h^{n-1} }{\tau_n}, v_h \right)_{L^2_h} + \int_\Omega \GRAD u_h^n \cdot \GRAD v_h \diff x = \int_\Omega f^n v_h \diff x, \qquad \forall v_h \in \Fespace.
\end{equation}

Existence and uniqueness of discrete solutions is trivially achieved. The main issue that motivates our work is to obtain enough a priori estimates so that a family of discrete solutions $\{u_h^\dt\}_{h>0,\dt>0}$ converges, in a suitable sense, to the renormalized solution to \cref{eq:TheEqn}. To achieve this we, first of all, recast our scheme as a perturbed version of the standard $dG(0)$-in-time scheme. Namely, we define $\calB_h^\dt : \frakX_h^\dt \times \frakY_h^\dt \to \Real$ as
\begin{equation}
\label{eq:DefOfB}
  \begin{aligned}
      \calB_h^\dt(v_h^\dt,w_h^\dt) &\coloneqq \left( v_h^\dt(0) , w_h^\dt(0) \right)_{\Ldeux} + \int_0^T\int_\Omega \GRAD v_h^\dt \cdot \GRAD w_h^\dt \diff x \diff t \\
        &+ \sum_{n=1}^\calN \left( \jump{v_h^\dt}_{n-1}, w_h^n \right)_{L^2_h},
  \end{aligned}
\end{equation}
and $\calF_h^\dt : \frakY_h^\dt \to \Real$ as
\begin{equation}
\label{eq:DefOfF}
  \calF_h^\dt( w_h^\dt ) \coloneqq \int_\Omega u_0 w_h^\dt(0) \diff x + \int_{Q_T} f w_h^\dt \diff x \diff t.
\end{equation}
Notice that, if $w_h^\dt = \{w_h^n\}_{n=0}^\calN$, then we may rewrite the previous expression as
\[
  \calF_h^\dt(w_h^\dt) = \int_\Omega \calP_h u_0 w_h^0 \diff x + \sum_{n=1}^\calN \tau_n \int_\Omega f^n w_h^n \diff x.
\]
In summary, we may rewrite \cref{eq:TheScheme} as: Find $u_h^\dt \in \frakX_h^\dt$ such that
\begin{equation}
\label{eq:DgTimeScheme}
  \calB_h^\dt( u_h^\dt, v_h^\dt) = \calF_h^\dt(v_h^\dt), \qquad \forall v_h^\dt \in \frakY_h^\dt.
\end{equation}
The equivalence is standard, and the only difference with the canonical $dG(0)$-in-time scheme lies in the mass-lumping of the jump terms.

\subsection{Conditional inf-sup stability of the mass lumped implicit Euler scheme}
\label{sub:InfSup}

Our goal here will be to prove an inf-sup condition for the bilinear form $\calB_h^\dt$. For the standard implicit Euler scheme this result can be found in \cite[Lemma 71.18]{MR4248810}; see also \cite{MR1921920,MR4246874}. To our knowledge this, simple yet useful, result is not available in the literature and may be of its own interest.

\begin{theorem}[inf-sup]
\label{thm:InfSup}
  Assume that the discretization parameters satisfy the following reverse CFL condition
  \begin{equation}
  \label{eq:ReverseCFL}
    h^2 \leq \frac1{4C_Q^2} \min_{n=1}^\calN \tau_n,
  \end{equation}
  where $C_Q$ is the constant in \cref{Eq:error-ML}. Then, we have
  \[
    \frac12 \| v_h^\dt \|_{\frakX_h^\dt} \leq \sup_{w_h^\dt \in \frakY_h^\dt} \frac{ \calB_h^\dt(v_h^\dt,w_h^\dt) }{ \| w_h^\dt \|_{\frakY_h^\dt} }.
  \]
\end{theorem}
\begin{proof}
  For the purposes of this proof define $\calA_h^\dt : \frakX_h^\dt \times \frakY_h^\dt \to \Real$ as
  \[
    \calA_h^\dt(v_h^\dt,w_h^\dt) \coloneqq \left( v_h^\dt(0) , w_h^\dt(0) \right)_{\Ldeux} + \int_0^T\int_\Omega \GRAD v_h^\dt \cdot \GRAD w_h^\dt \diff x \diff t
    + \sum_{n=1}^\calN \int_\Omega \jump{v_h^\dt}_{n-1} w_h^n \diff x.
  \]
  According to \cite[Lemma 71.18]{MR4248810} this bilinear form satisfies a uniform inf-sup condition. Namely, for every $h>0$ and all $\dt>0$
  \[
    \| v_h^\dt \|_{\frakX_h^\dt} \leq \sup_{w_h^\dt \in \frakY_h^\dt} \frac{ \calA_h^\dt(v_h^\dt,w_h^\dt) }{ \| w_h^\dt \|_{\frakY_h^\dt} }.
  \]
  Clearly then
  \[
    \| v_h^\dt \|_{\frakX_h^\dt} \leq \sup_{w_h^\dt \in \frakY_h^\dt} \frac{ \calB_h^\dt(v_h^\dt,w_h^\dt) }{ \| w_h^\dt \|_{\frakY_h^\dt} } + \sup_{w_h \in\frakY_h^\dt} \frac{  \calC_h^\dt(v_h^\dt,w_h^\dt) }{ \| w_h^\dt \|_{\frakY_h^\dt} },
  \]
  where
  \[
    \calC_h^\dt(v_h^\dt,w_h^\dt) \coloneqq \left| \sum_{n=1}^\calN \int_\Omega \left( \jump{v_h^\dt}_{n-1} w_h^n - \calL_h( \jump{v_h^\dt}_{n-1} w_h^n)\right) \diff x \right|.
  \]
  Using \cref{Eq:error-ML} we obtain then that
  \begin{align*}
    \calC_h^\dt(v_h^\dt,w_h^\dt) &\leq C_Q h  \sum_{n=1}^\calN \| \jump{v_h^\dt}_{n-1} \|_\Ldeux \| \GRAD w_h^n \|_\Ldeuxd
    \\
    &\leq C_Q \frac{h}{\sqrt{ \min_{n=1}^\calN \tau_n } } \left( \sum_{n=1}^\calN \| \jump{v_h^\dt}_{n-1} \|_\Ldeux^2 \right)^{1/2} \left( \sum_{n=1}^\calN \tau_n \| \GRAD w_h^n \|_\Ldeuxd^2 \right)^{1/2}
    \\
    &\leq C_Q \frac{h}{\sqrt{ \min_{n=1}^\calN \tau_n } } \| v_h^\dt \|_{\frakX_h^\dt} \| w_h^\dt\|_{\frakY_h^\dt}.
  \end{align*}
  Thus, under the assumed inverse CFL condition, the claimed inf-sup condition holds.
\end{proof}

\subsection{A priori estimates}
\label{sub:Apriori}

We now present the main a priori estimate that we will use to assert convergence of our numerical scheme.

\begin{theorem}[a priori estimates]
\label{thm:APrioriEst}
  Let $\{u_h^\dt \in \frakX_h^\dt\}_{h>0,\dt>0}$ denote the family of solutions to \cref{eq:TheScheme}. Then, this family satisfies \cref{eq:MainEst} with
  \[
    F = \| f \|_{L^1(Q_T)} , \qquad U = C_1 C_\calP\|u_0 \|_{L^1(\Omega)}.
  \]
\end{theorem}
\begin{proof}
  Fix $k>0$. Set, in \cref{eq:TheScheme}, $v_h = \tau_n \calL_h \Trunc_k u_h^n$. \Cref{assume:DMP} then yields
  \begin{equation}
  \label{eq:FirstEqApriori}
    \left(  u_h^n - u_h^{n-1}, \Trunc_ku_h^n \right)_{L^2_h} + \tau_n \int_\Omega |\GRAD \calL_h \Trunc_k u_h^n|^2 \diff x \leq \tau_n k \| f^n \|_{L^1(\Omega)}.
  \end{equation}
  
  Recall now that the mass lumped inner product can be rewritten as
  \[
    \left(  u_h^n - u_h^{n-1}, \Trunc_ku_h^n \right)_{L^2_h} = \sum_{\vertex \in \VertInt} \Trunc_ku_h^n(\vertex) \left( u_h^n(\vertex) - u_h^{n-1}(\vertex) \right) \int_\Omega \phi_\vertex \diff x.
  \]
  Next, the convexity of $\Theta_k$ and the fact that $\Theta_k'=\Trunc_k$ imply that, for every $\vertex \in \VertInt$, we have
  \[
    \Theta_k( u_h^n(\vertex) ) - \Theta_k( u_h^{n-1}(\vertex) ) \leq \Trunc_ku_h^n(\vertex) \left( u_h^n(\vertex) - u_h^{n-1}(\vertex) \right).
  \]
  In other words,
  \[
    \| \Theta_k( u_h^n ) \|_{L^1_h} - \| \Theta_k( u_h^{n-1} ) \|_{L^1_h} \leq \left(  u_h^n - u_h^{n-1}, \Trunc_ku_h^n \right)_{L^2_h}.
  \]
  Substitute this in \cref{eq:FirstEqApriori},   and add over $n$ to conclude that
  \begin{multline*}
    \| \calL_h \Theta_k(u_h^\dt) \|_{L^\infty(0,T;L^1(\Omega))} + \int_0^T \int_\Omega |\GRAD\calL_h \Trunc_k u_h^\dt |^2 \diff x \diff t \leq \\
      k \sum_{n=1}^\calN \tau_n \| f^n \|_{L^1(\Omega)} + \| \Theta_k( u_h^0 ) \|_{L^1_h}.
  \end{multline*}

  We finally invoke \cref{lem:ThetakVSmassLump} and \cref{eq:L1stabL2proj} to conclude
  \[
    \| \Theta_k( u_h^0 ) \|_{L^1_h}  \leq C_1 k \| u_h^0 \|_{L^1(\Omega)} = C_1 k \| \calP_h u_0 \|_{L^1(\Omega)} \leq C_1 C_\calP k \| u_0 \|_{L^1(\Omega)},
  \]
  which gives the value of $U$. Finally, using the definition of the discrete right hand side
  \[
    \sum_{n=1}^\calN \tau_n \int_\Omega | f^n | \diff x \leq \sum_{n=1}^\calN \int_\Omega \int_{I_n} |f| \diff t \diff x = \| f \|_{L^1(Q_T)}.
  \] 
  This defines the value of $F$ and finishes the proof.
\end{proof}

\subsection{Convergence}
\label{sub:Convergence}

We are now in position to state and prove the convergence of our numerical scheme.

\begin{theorem}[convergence]
  Suppose that $\{\Triang\}_{h>0}$ satisfies \cref{assume:DMP} and that the discretization parameters satisfy \eqref{eq:ReverseCFL}. Then, as $(h,\dt)\to (0,0)$, we have that, for every $q < \barq$,
  \[
    \| u - u_h^\dt \|_{L^\infty(0,T;L^1(\Omega))} + \| u - u_h^\dt \|_{L^q(0,T;W^{1,q}_0(\Omega))} \to 0,
  \]
  where $u$ is the renormalized solution to \eqref{eq:TheEqn}.
\end{theorem}
\begin{proof}
  Our method of proof draws inspiration from \cite[Theorem 3.2]{MR2266830}. Fix $\epsilon >0$. Let $\{(u_{0,m},f_m)\}_{m \in \Natural} \subset L^2(\Omega) \times L^2(Q_T)$ be a sequence such that, as $m \to \infty$,
  \[
    \| u_{0,m} - u_0 \|_{L^1(\Omega)} + \| f_m - f \|_{L^1(Q_T)} \to 0.
  \]
  
  Denote by $\{u_m\}_{m \in \Natural} \subset L^2(0,T;\Hunz) \cap H^1(0,T;\Hmun)$ the weak solutions to \cref{eq:TheEqn} with data $(u_{0,m}, f_m)$. The consistency of \cref{thm:Renormalized} shows that these are also renormalized solutions. Thus, the continuous dependence of renormalized solutions of \cref{thm:Renormalized} implies that there is $m_1 \in \Natural$ such that, for every $m \geq m_1$,
  \[
    \| u - u_m \|_{L^\infty(0,T;L^1(\Omega))} + \| \GRAD( u - u_m ) \|_{L^q(0,T;\bfL^q(\Omega))} < \frac\epsilon3.
  \]
  
  Let now, for $m\geq m_1$, $\{u_{m,h}^\dt\}_{h>0,\dt>0}$ denote the family of solutions to \eqref{eq:TheScheme} with data $(u_{0,m}, f_m)$. Since the discretization parameters are assumed to satisfy \cref{eq:ReverseCFL}, the inf-sup condition of \cref{thm:InfSup} holds. This immediately implies a C\'ea-type best approximation result, \ie
  \begin{multline*}
    \| u_m - u_{m,h}^\dt \|_{L^\infty(0,T;L^1(\Omega))} +  \| \GRAD( u_m - u_{m,h}^\dt) \|_{L^q(0,T;\bfL^q(\Omega))} \lesssim
    \\
    \| u_m - u_{m,h}^\dt \|_{L^\infty(0,T;L^2(\Omega))} + \| \GRAD( u_m - u_{m,h}^\dt )\|_{L^2(0,T;\bfL^2(\Omega))} \lesssim
    \\
    \inf_{w_h^\dt \in \frakX_h^\dt} \| \GRAD( u_m - w_h^\dt) \|_{\frakX_h^\dt},
  \end{multline*}
  where we also used that $q<\barq<2$ and the fact that the $\frakX_h^\dt$ norm controls the one in $L^\infty(0,T;\Ldeux)\cap L^2(0,T;\Hunz)$. Standard approximation properties of $\{\frakX_h^\dt\}_{h>0,\dt>0}$ can then be invoked to conclude that, for $h$ and $\dt$ sufficiently small, we have
  \[
    \| u_m - u_{m,h}^\dt \|_{L^\infty(0,T;L^1(\Omega))} +\| \GRAD( u_m - u_{m,h}^\dt) \|_{L^q(0,T;\bfL^q(\Omega))} < \frac\epsilon3.
  \]
  
  Next, by linearity, we realize that $e_{m,h}^\dt \coloneqq u_h^\dt - u_{m,h}^\dt$ solves \cref{eq:TheScheme} with data $(u_0-u_{0,m},f-f_m)$. The a priori estimate of \cref{thm:APrioriEst} then implies that the family $\{e_{m,h}^\dt\}_{m \in \Natural,h>0,\dt>0}$ satisfies \cref{eq:MainEst} with
  \[
    F = \| f-f_m \|_{L^1(Q_T)}, \qquad U =  C_1 C_\calP \| u_0 - u_{0,m} \|_{L^1(\Omega)}.
  \]
  In particular
  \begin{equation}
  \label{eq:AlmostLinfL1conv}
    \max_{n=1, \ldots,\calN} \int_\Omega \calL_h \Theta_1( e_{m,h}^n ) \diff x \leq \| f-f_m \|_{L^1(Q_T)} +  C_1 C_\calP \| u_0 - u_{0,m} \|_{L^1(\Omega)}.
  \end{equation}
  We may also invoke \cref{thm:Truncations} to conclude that, for $q<\barq$,
  \begin{equation}
  \label{eq:ConvLqW1q}
    \| \GRAD e_{m,h}^\dt \|_{L^q(0,T;\bfL^q(\Omega))} \lesssim \left[ \| f-f_m \|_{L^1(Q_T)} +  C_1 C_\calP \| u_0 - u_{0,m} \|_{L^1(\Omega)} \right]^{1/\barq}.
  \end{equation}
  This is almost what we need. All that is missing is a bound in $L^\infty(0,T;L^1(\Omega))$ for $\{e_{m,h}^\dt\}$ in terms of $F$ and $U$. Drawing inspiration from the proof of \cite[Claim 2]{MR1436364} we obtain it now. Let $n \in \{1,\ldots,\calN\}$ be arbitrary and we observe that
  \[
    \int_\Omega \calL_h \Theta_1( e_{m,h}^n ) \diff x = \sum_{\vertex \in \VertInt}  \Theta_1( e_{m,h}^n(\vertex) ) \int_\Omega \phi_\vertex \diff x.
  \]
  We split now the interior vertices into two disjoint sets:
  \[
    \VertInt(s,n) \coloneqq \left\{ \vertex \in \VertInt : |e_{m,h}^n(\vertex)| \leq 1 \right\}, \qquad \VertInt(b,n) \coloneqq \left\{ \vertex \in \VertInt : |e_{m,h}^n(\vertex)| > 1 \right\},
  \]
  and use that
  \[
    |s| \leq 1 \quad \implies \quad \Theta_1(s) = \frac{s^2}2, \qquad \qquad |s|>1 \quad \implies \quad \Theta_1(s) > \frac{|s|}2
  \]
  and \cref{eq:ErrEstMassLump} to estimate
  \[
    \frac12 \left[ \sum_{\vertex \in \VertInt(s,n) } |e_{m,h}^n(\vertex)|^2 \int_\Omega \phi_\vertex \diff x + \sum_{\vertex \in \VertInt(b,n) } |e_{m,h}^n(\vertex)| \int_\Omega \phi_\vertex \diff x \right] \leq \int_\Omega \calL_h \Theta_1( e_{m,h}^n) \diff x.
  \]
  We may now use \cref{prop:MassLump} to get
  \begin{align*}
    \| e_{m,h}^n \|_{L^1(\Omega)} &\leq  \int_\Omega \calL_h |e_{m,h}^n |\diff x
    \\
    &= \sum_{\vertex \in \VertInt(s,n) } |e_{m,h}^n(\vertex)| \int_\Omega \phi_\vertex \diff x + \sum_{\vertex \in \VertInt(b,n) } |e_{m,h}^n(\vertex)| \int_\Omega \phi_\vertex \diff x \eqqcolon \mathrm{S} + \mathrm{B}.
  \end{align*}
  For the first term a simple Cauchy-Schwarz inequality yields
  \begin{align*}
    \mathrm{S} &\leq \left( 2 \sum_{\vertex \in \VertInt(s,n)} \int_\Omega \phi_\vertex \right)^{1/2} \left( \frac12 \sum_{\vertex \in \VertInt(s,n)} |e_{m,h}^n(\vertex) |^2 \int_\Omega \phi_\vertex \right)^{1/2}
    \\
    &\leq \sqrt{2 |\Omega|} \left( \int_\Omega \calL_h \Theta_1( e_{m,h}^n) \diff x \right)^{1/2}.
  \end{align*}
  On the other hand, the bound on the second term is immediate, i.e.,
  \[
    \mathrm{B} \leq 2 \int_\Omega \calL_h \Theta_1( e_{m,h}^n) \diff x.
  \]
  We thus gather to obtain, since $n$ was arbitrary,
  \begin{multline*}
    \| e_{m,h}^\dt \|_{L^\infty(0,T;L^1(\Omega))} \leq
    \sqrt{2|\Omega|} \left( \max_{n = 1, \ldots, \calN } \int_\Omega \calL_h \Theta_1( e_{m,h}^n) \diff x \right)^{1/2} \\
    + 2\max_{n = 1, \ldots, \calN } \int_\Omega \calL_h \Theta_1( e_{m,h}^n) \diff x ,
  \end{multline*}
  which, combined with \cref{eq:AlmostLinfL1conv} finally yields
  \begin{multline}
  \label{eq:NowIsLinfL1Conv}
    \| e_{m,h}^\dt \|_{L^\infty(0,T;L^1(\Omega))} \lesssim
      \left( \| f-f_m \|_{L^1(Q_T)} +  C_1 C_\calP \| u_0 - u_{0,m} \|_{L^1(\Omega)} \right)^{1/2}
      \\
       + \| f-f_m \|_{L^1(Q_T)} + C_1 C_\calP \| u_0 - u_{0,m} \|_{L^1(\Omega)}.
  \end{multline}

  We can choose then $m_2 \geq m_1$ which will guarantee that, for $m \geq m_2$,
  \begin{multline*}
    C\left[
      \left( \| f-f_m \|_{L^1(Q_T)} +  C_1 C_\calP \| u_0 - u_{0,m} \|_{L^1(\Omega)} \right)^{1/2}
      + \| f-f_m \|_{L^1(Q_T)}
    \right.
    \\
    \left.
      +  C_1 C_\calP \| u_0 - u_{0,m} \|_{L^1(\Omega)}
      + \left(
        \| f-f_m \|_{L^1(Q_T)} + C_1 C_\calP \| u_0 - u_{0,m} \|_{L^1(\Omega)}
      \right)^{1/\barq}
    \right] < \frac\epsilon 3,
  \end{multline*}
  where $C>0$ is the constant induced by the one hidden in \eqref{eq:ConvLqW1q}.
  Therefore, \cref{eq:ConvLqW1q} and \cref{eq:NowIsLinfL1Conv} imply
  \[
    \| e_{m,h}^\dt \|_{L^\infty(0,T;L^1(\Omega))} + \| \GRAD e_{m,h}^\dt \|_{L^q(0,T;\bfL^q(\Omega))} < \frac\epsilon3.
  \]

  In conclusion, if $h$ and $\dt$ are small enough
  \[
    \| u - u_h^\dt \|_{L^\infty(0,T;L^1(\Omega))} + \| \GRAD( u - u_h^\dt ) \|_{L^q(0,T;\bfL^q(\Omega))} < \epsilon,
  \]
  and this shows convergence.
\end{proof}

\section{Conclusions,  extensions, and future research}
\label{sec:OtherStuff}

Having obtained a convergent scheme for the simplest parabolic equation possible, we briefly mention ways in which our results, without much effort, can be generalized.

\begin{itemize}[left=0pt]
  \item \textbf{Variable coefficients:} The equation in \cref{eq:TheEqn} may be generalized to
  \[
    \partial_t u -\DIV( \bfA \GRAD u )  = f.
  \]
  Here $\bfA \in L^\infty(\Omega;\Real^{d \times d})$ is symmetric, \ie $\bfA(x)^\top = \bfA(x)$ for almost every $x \in\Omega$, and there are constants $0 < \lambda \leq \Lambda$ such that
  \[
    \lambda |\bfv|^2 \leq \bfA(x)\bfv \cdot \bfv \leq \Lambda |\bfv|^2, \qquad \forall \bfv \in \Real^d, \ \mae \ x \in \Omega.
  \]
  The case $\bfA = a \bfI_d$, where $a \in L^\infty(\Omega)$, and $\bfI_d$ is the identity matrix merely requires adjusting the constants in our arguments. For the general case an analogue of \cref{assume:DMP} is needed. We refer the reader to \cite[Section 6]{MR2266830} for suitable mesh conditions.
  
  \item \textbf{Lower order terms:} Another related problem that can be tackled performing only minor variations to this approach is a parabolic reaction-diffusion equation. In fact, the mass-lumping approach described in this work can be extended to such a case, and the same results presented in the previous sections follow only after minor modifications of what was presented in this work.  A more significant modification of
  the analysis presented herein would be needed for the case a convection term is added to the formulation. In fact, in such a case the mesh requirements need to be much more strict (especially if 
  convection dominates), and the finite element method needs to be stabilized somehow.
  
  \item \textbf{Open questions:} Several problems remain open at the moment.  For example, the development of a finite element method that converges for any shape-regular family of triangulations without
  the need to modify the PDE first is an interesting, and challenging, problem.  In addition, the extension of the results presented in this work to nonlinear PDEs does not seem to be an easy task. These, and other
  problems will be the subject of future research.
\end{itemize}

\section*{Acknowledgements}

AJS has been partially supported by NSF grant DMS-2409918.
The core of this work was carried out while the authors were participating in the the Research-in-Groups programme: ``\emph{Approximation of renormalized solutions}'' at the International Centre for Mathematical Sciences, Edinburgh. ICMS support is gratefully acknowledged.

\bibliographystyle{siamplain}
\bibliography{BS}

\end{document}